\pdfoutput=1
\documentclass[12pt,a4paper,oneside]{amsart}
\usepackage{color,amssymb,latexsym,amsfonts,textcomp,fancyhdr,calc,graphicx}
\usepackage{pdfpages,amsfonts,amsthm,amssymb,latexsym,graphicx,leftidx,fancyhdr}
\usepackage{amsmath,amstext,amsthm,amssymb,amsxtra}
\usepackage[colorlinks,citecolor=red,hypertexnames=false]{hyperref}
\usepackage{dsfont}
\usepackage[framemethod=tikz]{mdframed}
\usepackage{asymptote}
\usepackage{enumerate}
\usepackage{bbm} 

\usepackage[left=3cm,right=3cm,bottom=2cm,top=2cm]{geometry} 

\usepackage[ULforem]{ulem}
\usepackage{cite}

\usepackage{txfonts,pxfonts,tikz} 
\usepackage{tgschola}
\usepackage[T1]{fontenc}

\theoremstyle{plain}
\newtheorem{theorem}{Theorem}[section]
\newtheorem{lemma}{Lemma}[section]

\newtheorem{proposition}{Proposition}[section]
\newtheorem{conjecture}[theorem]{Conjecture}

\newtheorem{question}{Question}

\theoremstyle{definition}
\newtheorem{definition}{Definition}[section]
\newtheorem{remark}{Remark}[section]

\newtheorem{case[theorem]}{Case}

\newcommand{\beql}[1]{\begin{equation}\label{#1}}
\newcommand{\eeq}{\end{equation}}
\newcommand{\comment}[1]{}

\newcommand{\Abs}[1]{{\left|{#1}\right|}}

\newcommand{\Norm}[1]{{\left\|{#1}\right\|}}

\newcommand{\Floor}[1]{{\left\lfloor{#1}\right\rfloor}}
\newcommand{\Ceil}[1]{{\left\lceil{#1}\right\rceil}}

\newcommand{\Set}[1]{{\left\{{#1}\right\}}}

\newcommand{\RR}{{\mathbb R}}

\newcommand{\CC}{{\mathbb C}}
\newcommand{\ZZ}{{\mathbb Z}}

\newcommand{\one}{{\mathbbm 1}}

\newcommand{\dens}{{\rm dens\,}}

\newcommand{\ft}[1]{\widehat{#1}}

\newcommand{\pp}[1]{{\left({#1}\right)}}

\newcommand{\diam}{{\rm diam\,}}

\newcounter{rem}
\newcommand{\Inn}{\operatorname{Int}}

\newcounter{step}
\makeatletter
\@namedef{subjclassname@2020}{\textup{2020} Mathematics Subject Classification}
\makeatother

\newcounter{mysec}
\newcounter{mysubsec}[mysec]
  \newenvironment{enumerate-math}
{\begin{enumerate}
		\addtolength{\itemsep}{5pt}
		}
	{\end{enumerate}}

\begin{document}

\title{Tiling, spectrality and aperiodicity of connected sets}
    
    	\author{Rachel Greenfeld}
    	\address{School of Mathematics, Institute for Advanced Study, Princeton, NJ 08540.}
    	\email{greenfeld.math@gmail.com}
    	\author{Mihail N. Kolountzakis}
    	\address{Dept. of Mathematics and Applied Mathematics, University of Crete, Voutes Campus, 70013 Heraklion, Crete, Greece.}
    	\email{kolount@uoc.gr}

\begin{abstract}
Let $\Omega\subset \RR^d$ be a set of finite measure.
    The periodic tiling conjecture suggests that if $\Omega$ tiles $\RR^d$ by translations then it admits at least one periodic tiling. Fuglede's conjecture suggests that $\Omega$ admits an orthogonal basis of exponential functions if and only if it tiles $\RR^d$ by translations. Both conjectures are known to be false in sufficiently high dimensions, with all the so-far-known counterexamples being highly disconnected.
    On the other hand, both conjectures are known to be true for convex sets.
    In this work we connectify counterexamples to the above conjectures, at a cost in dimension.

    (a) Starting from a counterexample to the periodic tiling conjecture in dimension $d$ we construct another counterexample, in dimension $d+2$, which is connected.

    (b) Then we extend our method and show, starting from a counterexample in dimension $d$ to the ``spectral $\implies$ tiling'' direction of the Fuglede conjecture, that connected such counterexamples exist in dimension $d+2$.

    (c) Last, we show that counterexamples to the ``tiling $\implies$ spectral'' direction of the Fuglede conjecture exist in {\em some} dimension, by appropriately iterating our method for the previous two problems.
\end{abstract}

\keywords{Tiling. Spectral sets. Aperiodic tiling. Einstein tiling problem}
\subjclass[2020]{42B10, 52C22, 52C23}

\maketitle

\tableofcontents

\section{Introduction}

\subsection{Trading dimension for freedom in tilings by translation}

Tiling by translation is a fascinating subject with connections to several parts of analysis and number theory, as well as, of course, geometry.
Restricting the motions of the tile to translations imposes a stronger structure on tilings compared with tilings where the tile (or tiles)
are allowed a greater group of motions. Tilings by translation often have, or are conjectured to have, properties that more general tilings do not have. This paper focuses on two of them: periodicity and spectrality. In the first we seek to understand if a translational tile must also be able to tile in a periodic manner, a property known to fail for tilings with a larger group of motions.
In the second the Fuglede conjecture identifies domains that tile with domains that admit an orthogonal basis of exponentials for their $L^2$ space.

It has turned out that both these properties cease to hold when the dimension is sufficiently large. It appears that the extra freedom afforded
by high dimension compensates for the rigidity imposed by restricting to translations. It is exactly this phenomenon that we exploit in this paper: increasing the dimension allows us to obtain more well behaved counterexamples to the Periodic Tiling Conjecture and to the Fuglede Conjecture, namely it allows us to obtain {\em connected} sets as counterexamples.

\subsection{Tilings and periodicity}

The study of the structure of tilings goes back to Hilbert's 18th problem. This problem was later generalized to the well known ``einstein\footnote{Here, the word ``einstein'' refers to ``one stone'' in German.} problem'', which asks about the existence of a single shape which tiles the space but does so only in a non-periodic way. Such a tile is called ``aperiodic'' or an ``einstein''. Socolar--Taylor \cite{einstein} constructed a planar aperiodic tile which tiles the plane by translations, rotations and also reflections, but this tile is highly disconnected. The Socolar--Taylor construction was later extended  to the Schmitt--Conway--Danzer tile: A convex three-dimensional domain which tiles $\RR^3$ aperiodically by translations, rotations and reflections  \cite{schmidt}.
The einstein problem for {\it planar connected tiles}  remained open, until very recently, when ``The Hat'' tile was discovered by Smith--Myers--Kaplan--Goodman-Strauss \cite{hat}. Moreover, in a subsequent paper, the same authors constructed a connected planar ``einstein'' which tiles the plane aperiodically by translations and rotations only (no reflections) \cite{chiral}. It is known, however, that there is no  {\it translational} einstein which is a topological disk \cite{BN,K}.  It was recently shown \cite{GT22} that aperiodic translational tiles exist in high dimensions. The first part of this paper (Section \ref{sec:aperiodic}) is devoted to the question whether there are any {\it aperiodic connected translational tiles}.

Let $\Omega\subset \RR^d$ be a measurable set of finite, positive measure. We call $\Omega$ a \emph{translational tile} of $\RR^d$ if there exists a (countable) set $A\subset \RR^d$ such that the family of translates of $\Omega$ along the elements of $A$:
$$\Omega+a, \;  a\in A,
$$
covers almost every point in $\RR^d$ exactly once. The set $A$ is then called a \emph{tiling of $\RR^d$ by $\Omega$}, and we write:
$$  \Omega\oplus A=\RR^d.
$$

Similarly, a finite subset $F\subset \ZZ^d$ is a translational tile of $\ZZ^d$ if there exists $A\subset \ZZ^d$ such that the sets $F+a$, $a\in A$, form a partition of $\ZZ^d$, namely: $F\oplus A=\ZZ^d$. In this case, $A$ is called a \emph{tiling of $\ZZ^d$ by $F$}.

For $G=\RR^d$ or $G=\ZZ^d$,
a tiling $A$ in $G$ is said to be \emph{periodic} if there exists a lattice $\Lambda$, a discrete subgroup of $G$ containing $d$ linearly independent elements, such that $A$ is invariant under translations by any point in this lattice; namely 
$$ A+\lambda=A,\quad \lambda\in\Lambda
$$
for some co-compact subgroup $\Lambda$ of $G$. A translational tile of $G$ is called \emph{aperiodic} if none of the tilings that it admits are periodic.

 \begin{remark}
We caution that in some of the literature the term ``periodic'' instead refers to sets that are unions of cosets of some non-trivial cyclic subgroup of $G$ (in other words, one period vector is sufficient, rather than a set of periods that ``span'' the group).  The notion of an aperiodic tiling  is similarly modified in such literature, and the notion of aperiodicity used here is sometimes referred to as ``weak aperiodicity''. For tilings in dimensions $d\leq 2$ the two notions of aperiodicity coincide \cite[Theorem 3.7.1]{GS}.
\end{remark}

In the 60's,  H. Wang \cite{wang} conjectured that any tiling, by an arbitrary finite number of tiles, in $\ZZ^2$ admits a periodic tiling. Wang also showed that if this conjecture were true, then the question whether a given collection of finite subsets of $\ZZ^2$ tiles would be \emph{algorithmically decidable}: there would be an algorithm that provides an answer to this question in finite time. A few years later, Berger proved  \cite{Ber,Ber-thesis} a negative answer to both questions. He constructed an \emph{aperiodic} tiling with 20,426 tiles: this tile-set admits tilings but none of these tilings are periodic. Then, using this construction, he also proved that tilings by multiple tiles in $\ZZ^2$ are undecidable. Since then, there has been an extensive effort to reduce the possible size of aperiodic and undecidable tile-sets, see \cite[Table 2]{GT21}. Recently, in \cite{GT21}, it was proved that tilings with two tiles are  undecidable in high dimensions and later, in \cite{GT23} the undecidability of translational monotilings was established.

As for translational tiling by a single tile, the celebrated {\it periodic tiling conjecture} \cite{GS,LW}  asserts that there are no aperiodic translational tiles:

\begin{conjecture}[The periodic tiling conjecture]\label{ptc}
    Let $\Omega\subset\RR^d$ be a set of finite, positive measure. If $\Omega$ tiles $\RR^d$ by translations then it must admit at least one periodic tiling. 
\end{conjecture}

 The periodic tiling conjecture  is known to hold in $\RR$ \cite{LW}, in $\RR^2$ for topological disks \cite{BN,K} and also for convex domains in all dimensions  \cite{V,M}. However, very recently the periodic tiling conjecture was disproved in high dimensions  \cite{GT22}. 

Since the counterexample constructed in \cite{GT22} is disconnected, a natural followup question  is whether the periodic tiling conjecture is true for  connected  sets\footnote{To avoid trivial constructions, e.g., adding  zero-measure line segments between connected components to make  the set connected while trivially preserving aperiodicity, we require that the connected set is also the closure of its interior.} in all dimensions, see \cite[Question 10.3]{GT22}. 

Our first result gives a negative answer to this question:

\begin{theorem}\label{main:aperiodic}
For sufficiently large $d$, there exists a  set $\Omega$ in $\RR^d$ of  finite measure which is the closure of its interior, such that:
\begin{enumerate-math}
\item  $\Omega$ is connected.
\item  $\Omega$ tiles $\RR^d$ by translations.
\item If $\Omega\oplus A=\RR^d$ then $A$ is non-periodic.
\end{enumerate-math}
\end{theorem}

In fact, we show that any $d$-dimensional disconnected counterexample to the periodic tiling conjecture $\Omega$ gives rise to a $(d+2)$-dimensional  counterexample $\Omega'$, which is {\it connected}. 

The proof is done by first showing that certain type of operations on a given finite set $F\subset \ZZ^d$ preserve aperiodicity, see Theorem \ref{thm:discrete}. This latter theorem is general, and might be of independent interest. Then, we use this theorem to construct  $(d+2)$-dimensional ``folded briges'' between the connected components of a given aperiodic tile $F\subset \ZZ^d$, while preserving its aperiodicity. Finally, we inflate the obtained $(d+2)$-dimensional aperiodic tile, to get an aperiodic connected tile in $\RR^{d+2}$.

\subsection{Tiling and spectrality} A measurable set $\Omega\subset \RR^d$ of positive,
 finite measure is called \emph{spectral} if there is a frequency set $\Lambda\subset \RR^d$ such that the system
$$E(\Lambda)\coloneqq \{e^{2\pi i \lambda\cdot x}\}_{\lambda\in\Lambda}
$$
constitutes an orthogonal basis for $L^2(\Omega)$.
In this case, the set $\Lambda$ is called \emph{a spectrum for $\Omega$}.

The study of spectral sets goes back to Fuglede \cite{fug}, who in 1974 considered the question of the existence of commuting extensions to $L^2(\Omega)$ of the partial differentiation operators defined on $C_c(\Omega)$, where $\Omega$ is an open set -- a question where the notion of spectral sets  arose naturally. He conjectured that spectral sets are exactly the ones which tile by translations:

\begin{conjecture}[Fuglede's spectral sets conjecture]\label{Fug}
An open set $\Omega\subset \RR^d$ of finite, positive measure is spectral if and only if it tiles space by translations.
\end{conjecture}

Fuglede's conjecture motivated an extensive study of the nature of the connection between the two properties: The analytic property of spectrality and the geometric property of tiling by translations. Throughout the years many positive results towards the conjecture have been obtained, see \cite[Section 4]{KM10} and the references mentioned there. In particular, the conjecture is known to hold for convex domains in all dimensions \cite{IKT,GL17,LM}. Nevertheless, in 2004, Tao discovered that there  exist counterexamples to Fuglede's conjecture. In \cite{T04}, he constructed  examples of 
sets $\Omega\subset\RR^d$, for any $d\geq 5$, which are spectral, but cannot tile by translations. Subsequently, by an enrichment of Tao's approach, examples of translational tiles which are not spectral  were also constructed, and eventually the dimension in these examples was reduced down to $d\ge 3$ \cite{KM06,KM2} (see \cite[Section 4]{KM10} for more references). All these examples arise from constructions of counterexamples to the {\it finite Abelian group} formulation of Fuglede's conjecture. Thus, when inflated to Euclidean space $\RR^d$, $d\geq 3$, each of  the known counterexamples is a finite union of  unit cubes centered at points of the integer lattice $\ZZ^d$. However, since in all the previously known examples the arrangement of the cubes is very sparse and disconnected, Fuglede's conjecture for connected sets remained open. In this paper we show that there are connected counterexamples to both directions of the conjecture. 

In Section \ref{sec:spectralnottile}, from a given disconnected set in $\RR^d$ which is spectral and does not tile, we construct  a connected set in $\RR^{d+2}$ which is spectral and does not tile:

\begin{theorem}\label{main:spectralnottile}
    For $d\geq 5$, there exists an open set $\Omega$ in $\RR^d$ of  finite measure, such that:
    \begin{enumerate-math}
   \item $\Omega$ is connected.
        \item $\Omega$ is spectral.
        \item  $\Omega$ does not tile $\RR^d$ by translations.
    \end{enumerate-math}
\end{theorem}

Similarly to the construction in Section \ref{sec:aperiodic}, this is done by constructing ``folded bridges'' in $\RR^{d+2}$ between the connected components of  a given spectral set in $\RR^d$ which is not a tile. We prove in Theorem \ref{thm:general} that this type of construction preserves spectrality as well as the tiling properties of the original set.

In Section \ref{sec:tilenotspectral}, we construct, from a given disconnected set  $\Omega\subset \RR^d$ which tiles and  is not spectral, a connected set in $\RR^{\tilde d}$, $\tilde d=\tilde d (\Omega) >d$, which tiles and  is not spectral:

\begin{theorem}\label{main:tilenotspectral}
    For sufficiently large $d$, there exists an open set $\Omega$ in $\RR^d$ of  finite measure, such that:
    \begin{enumerate-math}
    \item $\Omega$ is connected.
        \item $\Omega$ tiles $\RR^d$ by translations.
        \item  $\Omega$ is not spectral.
    \end{enumerate-math}
\end{theorem}

The proof is done by iteratively  constructing high dimensional ``spiral bridges'' between the connected components of $\Omega$, a given finite union of unit cubes which tiles and is not spectral. In Theorem \ref{thm:stacking} we prove that this type of construction preserves the non-specrality as well as the tiling properties of the original set.

Theorems \ref{thm:general} and \ref{thm:stacking} give a range of operations on a set that preserve its spectral and tiling properties. These theorems may, therefore, be of independent interest.

\subsection{Notation and preliminaries} Throughout this paper:
\begin{itemize}
    \item We denote the Euclidean norm by
$$\Norm{\,\cdot\,}\colon \RR^d\to [0,\infty).$$ 
\item  We denote the Lebesgue measure of a set $\Omega \subset \RR^d$ by $|\Omega|$, and for a set $F\subset \ZZ^d$, $|F|$ denotes the cardinality of $F$, or, equivalently, the counting measure of $F$.
\item For a number $r\in\RR$, $\Floor{r}\in \ZZ$ denotes the largest integer which is smaller or equal to $r$, and $\Ceil{r}$  denotes the smallest integer which is greater or equal to $r$.
\item For sets $A,B$ in a group $G$, we use the notation $A+B$ for Minkowski addition: $$\{a+b\colon a\in A,b\in B\}$$
of $A$ and $B$.  For  $A\subset G$ and $B\subset G'$ the set $A\times B \subset G\times G'$ is the Cartesian product:
$$\{(a,b)\colon a\in A,b\in B\}
$$
of $A$ and $B$.
\item For a function $f\colon \RR^d\to \CC$ we denote
$$\{f=0\}\coloneqq \{\xi \in\RR^d \colon f(\xi) =0 \}.
$$
\end{itemize}

\subsubsection{} Let $\Lambda\subset \RR^d$ be  a countable set and let $\Omega\subset \RR^d$ be measurable with positive, finite measure. Observe that the system $E(\Lambda)=\{e^{2 \pi i \lambda\cdot x}\}_{\lambda\in\Lambda}$ is orthogonal in $L^2(\Omega)$ if and only if 
\begin{equation}\label{eq:orthogonal}
            \pp{\Lambda -\Lambda} \setminus\{0\}\subset \{\ft{\one_\Omega}=0\}.
        \end{equation}

The \emph{upper   density of $\Lambda$} is defined as the quantity
$$\limsup_{R\to \infty} \sup_{x\in \RR^d} \frac{|\Lambda\cap (x + [-R/2,R/2]^d)|}{R^d}
$$
and the \emph{lower   density of $\Lambda$} is defined as 
$$\liminf_{R\to \infty} \sup_{x\in \RR^d} \frac{|\Lambda\cap (x + [-R/2,R/2]^d)|}{R^d}.
$$
If the upper density of $\Lambda$ is equal to its lower   density, we denote both quantities by $\dens{\Lambda}$ and say that $\Lambda$ has \emph{density} $\dens{\Lambda}$.

The following proposition is well known in the study of spectral sets. It will be used in the proofs of Theorems \ref{thm:general} and \ref{thm:stacking}.

\begin{proposition}\label{spectrality}
    Let $\Omega\subset \RR^d$ be a  measurable set  of positive, finite measure. The following are equivalent:
    \begin{enumerate-math}
        \item $\Omega$ is spectral.
        \item There exists   $\Lambda\subset \RR^d$ of upper density at least $|\Omega|$ such that \eqref{eq:orthogonal} is satisfied.
    \end{enumerate-math}
      Moreover, if $\Lambda\subset \RR^d$ is a spectrum for $\Omega$ then $\Lambda$ satisfies \eqref{eq:orthogonal} and $\dens{\Lambda}=|\Omega|$.
\end{proposition}

The proof of Proposition \ref{spectrality} follows by combining \cite[Section 3.1]{K04} or \cite[Lemma 3.1]{GL20} with \cite[Theorem 1]{K16}.

\subsection{Acknowledgment}
R.G.\ was  supported by the National Science Foundation  grants  DMS-2242871, DMS-1926686 and by the  Association of Members of the Institute for Advanced Study. M.K.\ was supported by the Hellenic Foundation for Research and Innovation, Project HFRI-FM17-1733 and by University of Crete Grant 4725. We thank Terence Tao for helpful suggestions to improve the exposition of the paper. We are grateful to Sha Wu of Hunan University for pointing out an error in the original ``folded bridge'' construction which led us to a much simplified ``folded bridge''. We thank the referee for helpful comments.

\section{Aperiodic connected tiles}\label{sec:aperiodic}

\begin{theorem}[Aperiodicity preserving operation]\label{thm:discrete}
    Let $F$ be a finite subset of $\ZZ^d$. Define the finite set
$$
X = \Set{(v_j,s_j): j=0, 1, \ldots, n-1} \subseteq \ZZ^{d+k}
$$
where $v_0,\dots,v_{n-1}\in\ZZ^d$ are arbitrary and $s_0,\dots,s_{n-1}$ are $n$ distinct points in  $\ZZ^k$ such that $$S=\{s_j\colon j=0, 1, \ldots, n-1\}$$ tiles $\ZZ^k$ by translations. Let $F' = \left( F\times \{0\}^k \right) \oplus X$.  Then  $F'$ is an aperiodic tile in $ \ZZ^{d+k}$ if $F$ is an aperiodic tile of $\ZZ^d$.
 \end{theorem}

\begin{remark}\label{lattice}
If $H, K$ are subgroups of $G$ then $[H:H\cap K] \le [G:K]$. This implies that if $\Lambda \subseteq \ZZ^m\times\ZZ^n$ is a lattice then $\Lambda \cap \ZZ^m\times\Set{0}^n$ is a lattice in $\ZZ^m\times\Set{0}^n$.
\end{remark}

\begin{proof}[Proof of Theorem \ref{thm:discrete}.]

Suppose that $F \subset \ZZ^d$ is an aperiodic tile. 
Suppose, towards a contradiction, that $F'$ is not aperiodic. Clearly, $F'$ tiles $\ZZ^{d+k}$. Indeed, by assumption, there is a tiling $A\subset \ZZ^d$ of $\ZZ^d$ by $F$, and a tiling $T\subset\ZZ^k$  of $\ZZ^{k}$ by $S$; by construction of $F'$ we then have that 
$$A'=A\times  T
$$ is a tiling of $\ZZ^{d+k}$ by $F'$.
Therefore, our assumption that $F'$ is not aperiodic implies that there exists a periodic tiling $A'$ of $\ZZ^{d+k}$ by $F'$  with period lattice $G' \subseteq \ZZ^{d+k}$.
Define $V = \ZZ^d\times \Set{0}^k$ and
$$
G = G' \cap V.
$$
It follows from Remark \ref{lattice} that $G$ is a lattice in $V$. Define also the subset of $V$
\begin{equation}\label{Omegatiling}
    A \coloneqq (A' + X) \cap V. 
\end{equation}
Since for every $a'\in A'$, $x\in X$ with $a'+x\in V$ and every $g\in G$ we have $a'+x+g = (a'+g)+x = a''+x$ for some $a'' \in A'$, we conclude that $A+G=A$, so that $A$ is periodic in $V$. Thus, to arrive at a contradiction, it is enough to prove that $ F \times \{0\}^k\oplus A=V$ is a tiling.
Observe that for every $a'\in A'$
$$
(F'+a') \cap V = (F\times\{0\}^k+X+a') \cap V = F\times\{0\}^k + \bigl( (X+a') \cap V \bigr),
$$
since $F\times \{0\}^k \subseteq V$. Thus, since  $$(F'+a') \cap V, \quad a'\in A'$$ form a tiling of $V$, so do the translates of $F \times\{0\}^k$ by all the points $x+a' \in V$, with $x\in X, a'\in A'$, which is exactly the set of translates $A$ defined in \eqref{Omegatiling}. 
\end{proof}

\begin{definition}\label{def:discreteconnected}
    Let $B\subset \ZZ^d$. A \emph{connected component} of $B$ is a subset $C$ of $B$ such that $C+[0,1]^d$  is a connected component of $B+ [0,1]^d $ in $ \RR^d$.
    
    If $B$ has a single connected component, we say that $B$ is \emph{connected}.
\end{definition}

\begin{remark}\label{connectivity}
In our definition two points $a, b \in \ZZ^d$ are connected to each other if and only if $\Abs{a_i-b_i} \le 1$, for all $i=1,2,\ldots,d$. In other words each point in $\ZZ^d$ has $3^d-1$ neighbors.

We could strengthen the notion of connectivity for subsets of $\ZZ^d$ to demand a, so-called, $2d$-connected path from any point of the set to any other (such a path is allowed to go from any point $x \in \ZZ^d$ to any of its $2d$ neighbors along the $d$ coordinate axes). Everything in this paper would work essentially the same.
\end{remark}

\subsection{Folded bridge construction in $\ZZ^d$}
Let $F\subset \ZZ^d$ be finite with $m+1>1$ connected components $C_0, C_1,\dots,C_m$.
Pick $m+1$ points $a_j \in C_j$ with $a_0=0$ for simplicity. Then there exists a path $v_0, \ldots, v_{n-1} \in \ZZ^d$, where
each $v_j$ is a neighbor of $v_{j\pm 1}$,
and
$$
v_0 = a_0 = 0,\ \ v_{n-1} = a_m
$$
and each $a_j$, $j=0, 1, \ldots, m$, belongs to the path
$$
\gamma:\ \ v_0, v_1, \ldots, v_{n-1}.
$$
Thus the path $v_j$ connects all connected components of $F$.
See Fig. \ref{fig:path}.

\begin{figure}[ht]
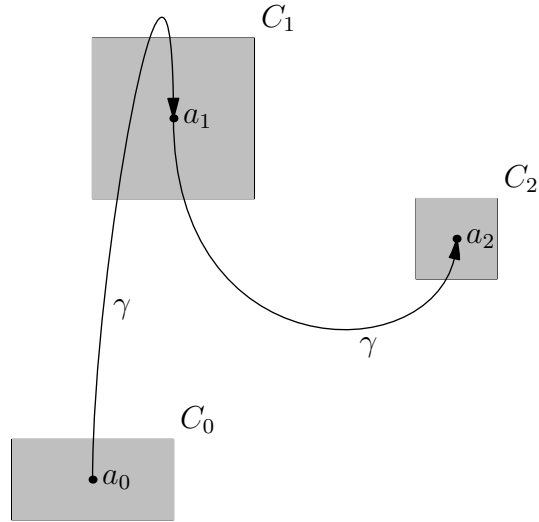

\centering
\begin{asy}
size(7cm);

int i, j, k, NN=10;

path[] P={(0, 0)--(2, 0)--(2, 1)--(0, 1)--cycle,
	(1, 4)--(3, 4)--(3, 6)--(1, 6)--cycle,
	(5, 3)--(6, 3)--(6, 4)--(5, 4)--cycle};
pair[] a={(1, 0.5), (2, 5), (5.5, 3.5)};
	
for(i=0; i<P.length; ++i) {
 draw(format("$C_{
 fill(P[i], mediumgray);
 dot(format("$a_{
}

/*
draw(Label("$\gamma$", Relative(0.3)), a[0]{up}..{down}a[1], Arrow);
draw(Label("$\gamma$", Relative(0.7)), a[1]{down}..{up}a[2], Arrow);
*/
draw(Label("$\gamma$", Relative(0.3)), a[0]..a[1], dotted, Arrow);
draw(Label("$\gamma$", Relative(0.7)), a[1]..a[2], dotted, Arrow);
\end{asy}
 \caption{The path $\gamma$, consisting of the points $v_0, \ldots, v_{n-1}$ visits all connected components of $F$.}
 \label{fig:path}
\end{figure} 



Define the sequence $S = \Set{s_j:\  j=0, 1, \ldots, 2n-1} \subseteq \ZZ^2$, as follows.
\begin{align}
\begin{split}\label{ss}
\, & s_0 = (0, 0), s_1 = (1, 0), \ldots, s_{n-1} = (n-1, 0),\\
 & s_n = (n-1, 1), s_{n+1} = (n-2, 1), \ldots, s_{2n-2}=(1, 1), s_{2n-1} = (0, 1).
\end{split}
\end{align}
as in Figure \ref{fig:snakes}.

\begin{figure}[ht]
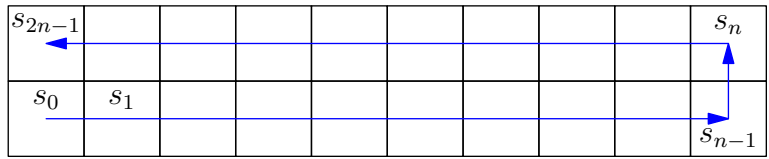

\centering
\begin{asy}
size(10cm);

int i, j, k, NN=10, m=2;

for(k=0; k<m-1; ++k) {
for(i=0; i<NN; ++i) draw(box((i, 2*k), (i+1, 2*k+1)));
for(i=0; i<NN; ++i) draw(box((i, 2*k+1), (i+1, 2*k+2)));
draw((1/2, 2*k+1/2)--(NN-1/2, 2*k+1/2), blue, Arrow);
draw((NN-1/2, 2*k+1/2)--(NN-1/2, 2*k+3/2), blue, Arrow);

draw((NN-1/2, 2*k+3/2)--(1/2, 2*k+3/2), blue, Arrow);
if(k<m-2) draw((1/2, 2*k+3/2)--(1/2, 2*k+5/2), blue, Arrow);
}
label("$s_0$", (0.5, 0.5), N);
label("$s_1$", (1.5, 0.5), N);
label("$s_{n-1}$", (NN-1/2, 0.5), S);
label("$s_n$", (NN-1/2, 1.5), N);
label("$s_{2n-1}$", (0.5, 1.5), N);
\end{asy}
\caption{The sequence $S=\Set{s_0,\dots,s_{2n-1}}\subseteq \ZZ^2$. There are two rows in this array, each of length $n$.}
 \label{fig:snakes}
\end{figure} 
From $F \subseteq \ZZ^d$ we construct the set $F' \subseteq \ZZ^{d+2}$ by 
$$
F' =  F \times \{0\}^2 + X
$$
where
\begin{align*}
X &= \Set{X_0, X_1, \ldots, X_{2n-1}}\\
 &= \{(0,s_0),(0,s_1),(0,s_2),\dots,(0,s_{n-1}),\\
 	&\ \ \ \ \ \ \ \ \ (v_0, s_n), (v_1, s_{n+1}), \ldots, (v_{n-1}, s_{2n-1})\}.
\end{align*}
Notice that this is a disjoint sum since the $s_j$ are all different
(so that $\Abs{F'} = \Abs{F} \cdot 2n$).

\begin{lemma}\label{lm:X}
The set $X$ is connected in $\ZZ^{d+2}$.
\end{lemma}
\begin{proof}
We first observe that for $j=0, 1, \ldots, n-2$ the point $X_j=(0, j, 0)$ is connected to $X_{j+1} = (0, j+1, 0)$ since they only differ at one coordinate and only by 1. We also have that $X_{n-1}=(0, n-1, 0)$ is connected to $X_n = (0, n-1, 1)$ (remember $v_0 = a_0 = 0$) since they only differ at the last coordinate by 1. Finally, if $j\ge n$ then $X_j=(v_{j-n}, j, 1)$ is connected to $X_{j+1}=(v_{j-n+1}, j+1, 1)$ since their first $d$ coordinates form two connected points in $\ZZ^d$ (since $v_{j-n}$ is connected to $v_{j-n+1}$) and they also differ by 1 at the $d+1$ coordinate.
\end{proof}

We imagine a copy of $\ZZ^d$ ``hanging'' from each of the $2n$ cells in Figure \ref{fig:snakes}, and, as we move from left to right and then left again, the copy of $F$ in that copy of $\ZZ^d$ is translated by the  vectors $\underbrace{0, \ldots, 0}_n, v_0, v_1, \ldots, v_{n-1}$.

%
%

We call this construction a ``folded bridge'' between the connected components $C_0, C_2, \ldots, C_m$ of $F$, giving $F'$. 
See Figure  \ref{fig:folding} for a visual illustration of the notion for the case $m=2$ (three connected components).

%
%
%

\begin{figure}[ht]
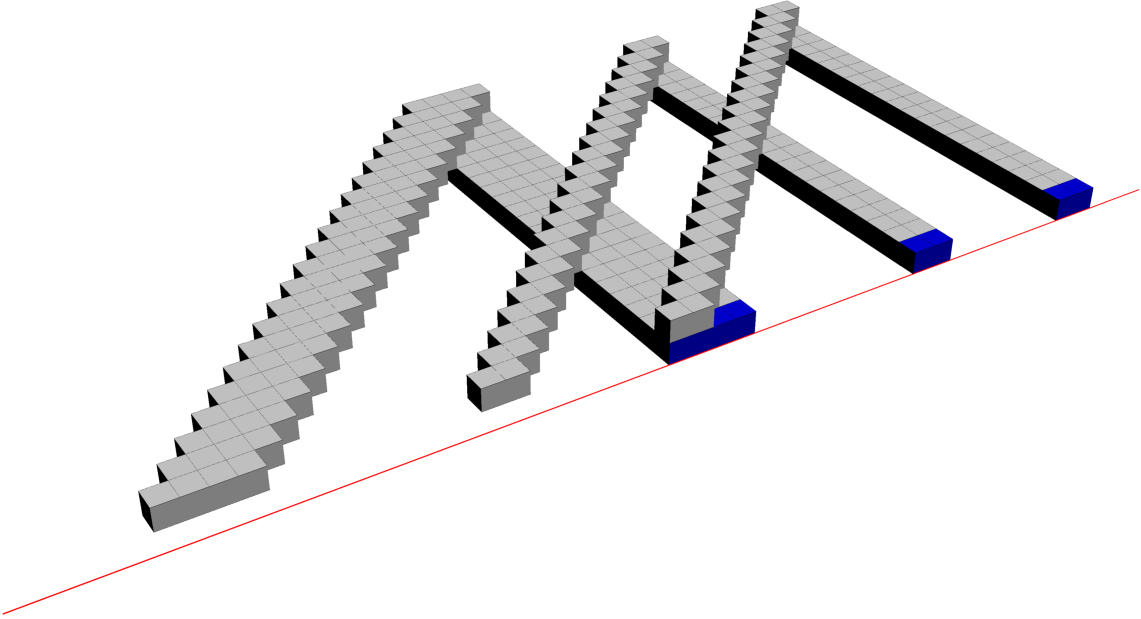

\centering
\begin{asy}
import three;
size(15cm);
currentprojection = perspective(-8, -10, 10);

void thick(triple p, triple u, triple v, real thickness, pen color) {
    path3 c=p--p+u--p+u+v--p+v--cycle;
    draw(c, black+opacity(1)+linewidth(0.02));
    draw(surface(c),color);
    triple n=scale3(thickness)*unit(cross(u, v));
    c=shift(n)*c;
    draw(c,black+opacity(1)+linewidth(0.02));
    draw(surface(c),color);

    path3 t=p--p+n--p+n+u--p+u--cycle;
    draw(surface(t), color);
    draw(surface(shift(v)*t), color);

    path3 t=p--p+n--p+n+v--p+v--cycle;
    draw(surface(t), color);
    draw(surface(shift(u)*t), color);
}

int[] T={0, 1, 2, 3, 12, 13, 20, 21}; // 1d tile
int[] A={0, 12, 20}; // one point from each connected component of T, in increasing order
int[] V={}; // path connecting each connected component to the next
int m=A.length; // number of connected components

for(int i=0; i<m-1; ++i) // compute V
 for(int j=A[i]; j<A[i+1]; ++j)
  V.push(j);
V.push(A[m-1]);
int n=V.length;

//triple p=(T[4]+1/2, 0, 1), q=p+(1, -1, 1);
//triple p=(T[4]+1/2, 0, 1), q=p+(1, -1, 1);
draw((-25, 0, 0)--(25, 0, 0), red);

/*
for(int i; i<m; ++i)
 draw("$F$", align=S, g=(T[A[i]], -0.1, 0)--(T[3]+1, -0.1, 0), black+linewidth(1));
draw("$F$", align=S, g=(T[4], -0.1, 0)--(T[5]+1, -0.1, 0), black+linewidth(1));
*/

triple u=(1, 0, 0), v=(0, 1, 0);//, start=(0, 0, 0);

for(int j=0; j<n; ++j) {
 pen P=palegray+opacity(1);
 if(j==0) P=blue+opacity(1);
 for(int i=0; i<T.length; ++i) {
  thick((T[i], j, 0), u, v, 1, P);
 }
}

for(int j=0; j<n; ++j) {
 pen P=palegray+opacity(1);
 for(int i=0; i<T.length; ++i) {
  thick((T[i]-V[j], n-1-V[j], 1), u, v, 1, P);
 }
}
\end{asy}

\caption{How $F'$ is constructed from $F$. A folded bridge on the set $F$ (blue is $F\times[0,1]^2$) connecting its three connected components. The red line is the ambient space for $F$, namely $\ZZ^d$.}
\label{fig:folding}
\end{figure}

\begin{lemma}\label{lm:F}
$F'$ is connected in $\ZZ^{d+2}$.
\end{lemma}
\begin{proof}
We first observe that
$$
F' = F\times \{0\}^2 +X = \bigcup_{j=0}^m (C_j\times \{0\}^2 + X),
$$
and each $C_j\times \{0\}^2 +X$ is connected from Lemma \ref{lm:X} and the fact that the sum of two connected sets is connected. It remains to show that the connected sets $C_j\times \{0\}^2 +X$ connect to each other as well. We show that for $j\ge 1$ the set $C_j\times \{0\}^2 +X$ connects to $C_0\times \{0\}^2 +X$. Indeed, there exists $k \in \Set{0, 1, \ldots, n-1}$ such that $a_j = v_k$ (by the construction of the path $v_j$, $j=0,\dots,n-1$). Then (recall that $a_0=0$)
$$
(a_j, s_{n+k}) = (a_0+v_k, s_{n+k})  \in C_0\times \{0\}^2 +X
$$
and
$$
(a_j, s_{n-k-1}) \in C_j\times \{0\}^2 +X.
$$
These two points have the same first $d+1$ coordinates and differ only in the last coordinate where the first point has 1 and the second has 0. (The point $s_{n+k}$ is right above $s_{n-k-1}$ in Fig.\ \ref{fig:snakes}.)
\end{proof}

By Theorem \ref{thm:discrete} we have that  $F'$ is aperiodic in $\ZZ^{d+2}$ if $F$ is aperiodic in $\ZZ^d$.

Using this, we can finally prove Theorem \ref{main:aperiodic}:

\begin{proof}[Proof of Theorem \ref{main:aperiodic}.]
By \cite[Corollary 1.5]{GT22}, if $d$ is sufficiently large, we can choose a finite $F \subset \ZZ^d$ which is an aperiodic translational tile. By applying the ``folded bridge'' construction above  we obtain a set $F' \subset \ZZ^{d+2}$ which is connected,  and is also an aperiodic translational tile (by Theorem \ref{thm:discrete}, since $S$ is a rectangle). Let $R_{d+2}$ be the ``dented $(d+2)$-dimensional cube''  constructed in the proof of  \cite[Lemma 2.2]{GT22}.  Observe that by construction of $R_{d+2}$, the set $F'+ R_{d+2}\subset \RR^{d+2}$ is connected if and only if $F'+[0,1]^{d+2}\subset \RR^{d+2}$ is connected; thus, since $F'$ is connected in $\ZZ^{d+2}$ in the sense of Definition \ref{def:discreteconnected}, $F'+ R_{d+2}$ is  connected in $\RR^{d+2}$. Moreover, the argument in the proof of \cite[Theorem 2.1]{GT22} gives that $F'+R_{d+2}$ is aperiodic in $\RR^{d+2}$, since $F'$ is aperiodic in $\ZZ^{d+2}$.
Finally, note that $F'+ R_{d+2}\subset \RR^{d+2}$ is equal to the closure of its interior.  
Theorem \ref{main:aperiodic} now follows, with $\Omega$ being $F' +R_{d+2}$.
\end{proof}

\section{Connected spectral sets that do not tile}\label{sec:spectralnottile}

The ultimate goal of this section is to prove Theorem \ref{main:spectralnottile}. We begin with the following general theorem, which shows that certain operations allow to construct, from a given spectral set $\Omega$, other sets that are spectral as well and  that  preserve the tiling property of the original set $\Omega$.

\begin{theorem}[Spectrality and tiling preserving operations]\label{thm:general}
    Let $\Omega$ be a bounded, measurable set in $\RR^d$. Define the finite set
$$
X = \Set{(v_j,s_j): j=0, 1, \ldots, n-1} \subseteq \RR^{d+k}
$$
where $v_0,\dots,v_{n-1}\in\RR^d$ and $s_0,\dots,s_{n-1}$ are $n$ distinct points in  $\ZZ^k$. Let $$S=\{s_j\colon j=0, 1, \ldots, n-1\}$$ and $\Omega' = \left( \Omega\times [0,1]^k \right) \oplus X$.  Then:
\begin{enumerate-math}
    \item Suppose that $S$ tiles $\ZZ^d$ by translations. Then $\Omega'$ tiles $\RR^{d+k}$ by translations if and only if $\Omega$ tiles $\RR^d$ by translations.
    \item If $\Omega\subset \RR^d$ and $S+[0,1]^k\subset\RR^k$ are spectral, then $\Omega'$ is spectral in $\RR^{d+k}$.
\end{enumerate-math}
\end{theorem}

 \begin{proof}[Proof of Theorem \ref{thm:general} (i).]
 If $A\oplus \Omega=\RR^d$  then $A'\oplus\Omega'=\RR^{d+k}$, where 
 $$A'=A\times T
 $$
 and $T\subset\ZZ^d$ is a tiling of $\ZZ^k$ by $S$. 
Conversely, if $\Omega'\oplus A'=\RR^{d+k}$ then, by a similar argument as in the proof of Theorem \ref{thm:discrete}, the set \eqref{Omegatiling} is a tiling of $\RR^d\times\{0\}^k$ by $\Omega\times \{0\}^k$.
\end{proof}

\begin{proof}[Proof of Theorem \ref{thm:general} (ii).]
Let $\Lambda\subset \RR^d$ be a spectrum for $\Omega$ and $\Sigma\subset \RR^k$  be a spectrum for $S+[0,1]^k$, then the Cartesian product set
$$
\Lambda':=\Lambda\times \Sigma=\{(\lambda,\sigma)\in \RR^{d+k}\colon \lambda\in \Lambda,\; \sigma\in\Sigma \}
$$
 defines an orthogonal system 
 $$E(\Lambda')=\{e^{2\pi i \lambda'\cdot x}\}_{\lambda'\in\Lambda'}
 $$
 in $L^2(\Omega')$. Indeed, let $\tau=(\lambda,\sigma),\tau'=(\lambda',\sigma')$ be distinct points in $\Lambda'$. By \eqref{eq:orthogonal}, we need to show that 
 \begin{equation}\label{eq:ftZeros}
     \ft{\one_{\Omega'}}(\tau'-\tau)=0.
 \end{equation}
 Observe that by the definition of $\Omega'$ we have 
$$\one_{\Omega'}(w_1,\dots,w_{d+k})=\one_\Omega(w_1,\dots,w_d)\one_{[0,1]^k}(w_{d+1},\dots,w_{d+k})*\left(\sum_{x\in X}\delta_x\right)(w_1,\dots,w_{d+k}).
$$
Therefore
 \begin{equation}\label{eq:ftOmega'}
     \ft{\one_{\Omega'}}(\xi_1,\dots,\xi_{d+k})=\ft{\one_\Omega}(\xi_1,\dots,\xi_d)\ft{\one_{[0,1]^k}}(\xi_{d+1},\dots,\xi_{d+k})\left( \sum_{x\in X}e^{2\pi i x\cdot(\xi_1,\dots,\xi_{d+k})}\right).
 \end{equation}
 If $\lambda', \lambda$ are distinct in $\Lambda$, then by \eqref{eq:orthogonal} $$\ft{\one_\Omega}(\lambda'-\lambda)=0,$$ 
 since $\Lambda$ is a spectrum for $\Omega$, and so, in particular $E(\Lambda)$ is orthogonal in $L^2(\Omega)$. Thus, in this case by \eqref{eq:ftOmega'} we see that \eqref{eq:ftZeros} is satisfied.  Otherwise, $\lambda'-\lambda =0$ and  $\sigma,\sigma'$ are distinct in the spectrum $\Sigma$ of $S+[0,1]^k$, so by \eqref{eq:orthogonal}
 $$\ft{\one_{S+[0,1]^k}}(\sigma'-\sigma)=0.
 $$
 By \eqref{eq:ftOmega'} we then have:
 \begin{align*}
     \ft{\one_{\Omega'}}(\tau'-\tau) &=\ft{\one_\Omega}(0)\ft{\one_{[0,1]^k}}(\sigma'-\sigma)\left(\sum_{x\in X}e^{2\pi i x\cdot(0,\sigma'-\sigma)}\right)\\
     &=|\Omega|\ft{\one_{[0,1]^k}}(\sigma'-\sigma)\left(\sum_{s\in S}e^{2\pi i s\cdot(\sigma'-\sigma)}\right)\\
     &= |\Omega|\ft{\one_{S+[0,1]^k}}(\sigma'-\sigma)=0.
 \end{align*}
Therefore \eqref{eq:ftZeros} is satisfied in this case as well, and hence $E(\Lambda')$ is orthogonal in $L^2(\Omega')$, as claimed.
 Now,  observe that $$\dens{\Lambda'}=\dens{\Lambda\times\Sigma}=\dens{\Lambda}\cdot \dens{\Sigma}.$$
 Thus, as $\Lambda$ is a spectrum for $\Omega$ and $\Sigma$ is a spectrum for $S+[0,1]^k$, by Proposition \ref{spectrality}, we have $$\dens{\Lambda'}=|\Omega||S+[0,1]^k|= |\Omega|n= |\Omega'|.$$
A further application of Proposition \ref{spectrality} then gives  that $\Omega'$ is spectral.
\end{proof}

\subsection{Folded bridge construction in $\RR^d$}
Let $\Omega$ be a bounded, open  set in $\RR^d$ with $m+1 >1$ connected components $C_0, C_1, \dots, C_m$.
Pick $m+1$ points $a_j$ in the interior of $C_j$ each and assume for simplicity $a_0=0$.
Let $K$ be large enough so that  we have
\beql{overlap}
C_i \cap (C_i+\delta_j) \neq \emptyset \; \text{ for all } i, j,
\eeq
where $\delta_i = \frac1{K}(a_{i+1}-a_i)$, $i=0,\ldots,m-1$. (In particular, since $\Omega$ is open, each of the sets $C_i$, $i=1,\dots,m-1$ is open, and thus each of these intersections has non-empty interior.)
Let $n=mK+1$ and
define the sequence $v_j$, $j=0, 1, 2, \ldots, n-1$, to consist of the $n$ values
\begin{align*}
\ & a_0, a_0+\delta_0, a_0+2\delta_0, \ldots, a_0+(K-1)\delta_0,\\ 
& a_1, a_1+\delta_1, a_1+2\delta_1, \ldots, a_1+(K-1)\delta_1,\\ 
& a_2, a_2+\delta_2, a_2+2\delta_2, \ldots, a_2+(K-1)\delta_2,\\
& \ \ \ \ \ \ \ \ \ \ \ \ \cdots\\
& a_{m-1}, a_{m-1}+\delta_{m-1}, a_{m-1}+2\delta_{m-1}, \ldots, a_{m-1}+(K-1)\delta_{m-1},\\
& a_m
\end{align*}
or: $$v_j=a_{\tilde j}+(j-K\tilde j)\delta_{\tilde j},$$ where $\tilde j=\Floor{\frac{j}{K}}$,
so that, in particular, all points $a_0, a_1, \ldots, a_m$ belong to the sequence $v_j$, $j=1,\dots,n-1$.
We then define 
$$
\Omega_1 =  \Inn\left(\Omega\times [0,1]^2  + X\right)
$$
where
\begin{align*}
	X &= \Set{X_0, X_1, \ldots, X_{2n-1}}\\
	&= \{(0,s_0),(0,s_1),(0,s_2),\dots,(0,s_{n-1}),\\
	&\ \ \ \ \ \ \ \ \ (v_0, s_n), (v_1, s_{n+1}), \ldots, (v_{n-1}, s_{2n-1})\},
\end{align*}
where $s_j \in \RR^2$ is the sequence defined in \eqref{ss} and shown in Figure \ref{fig:snakes}.
Notice that this is a disjoint sum up to measure zero  since the $s_j$ are all different. Clearly, the set $S=\{s_j\colon j=0,\dots, 2n-1\}$ tiles $\ZZ^2$ by translations and $S+[0,1]^2$ is spectral in $\RR^2$.

Let us now see why the set $\Omega_1$ is connected.
The first observation is that for every $\omega_1,\omega_2\in\Omega$ and $x_1,x_2\in [0,1]^2$ such that $(\omega_1, x_1), (\omega_2, x_2)\in \Inn\left(\Omega\times[0, 1]^2 \right)$
\beql{connected}
\omega_1, \omega_2 \text{ connected in } \Omega \implies
(\omega_1, x_1), (\omega_2, x_2) \text{ connected\footnotemark \ in } \Inn\left(\Omega\times[0, 1]^2 \right).
\eeq \footnotetext{We say that two points are connected in a set if they both belong to the same connected component of the set.}
When moving from one cell of Figure \ref{fig:snakes} to the next, the set
$$\Inn\left(\left(C_i\times[0, 1]^2 + (v_j, s_j)\right)\cup\left( C_i\times[0, 1]^2 + (v_{j+1}, s_{j+1})\right)\right)$$
 is connected, by \eqref{overlap} and the fact that $s_j$ and $s_{j+1}$ differ in one coordinate only and exactly by 1, so $s_j-s_{j+1} \in [-1, 1]^2$, the latter set being the difference set of $[0, 1]^2$.

Hence, when we move across one cell in Figure \ref{fig:snakes}, following the path, the connected components are either maintained or merging, so new connected components are not created along the way. Merging happens when we are moving on the upper row (see an illustration in Figure \ref{fig:folding}). Take $j\ge 1$ and let $v_k$ be such that $a_j = v_k$. Then
$$
(a_j, s_{n+k}+(\tfrac{1}{2},\tfrac{1}{2})) = (a_0+v_k, s_{n+k}+(\tfrac{1}{2},\tfrac{1}{2})) \in \Inn\left( C_0\times [0,1]^2+X \right)$$
and
$$
(a_j, s_{n-k-1}+(\tfrac{1}{2},\tfrac{1}{2})) \in \Inn\left( C_j\times [0,1]^2 +X\right).
$$
At that point the set $\Inn\left( \left(C_j\times [0,1]^2+X \right)\cup\left(C_0\times [0,1]^2+X\right)\right) $ gets connected, so, in the end we are left with one connected set.


By Theorem \ref{thm:general} since the boundary of $\Omega$ has measure zero) we have:
\begin{enumerate-math}
    \item $\Omega_1$ tiles $\RR^{d+2}$ by translations if and only if $\Omega$ tiles $\RR^d$ by translations.
    \item $\Omega_1$ is spectral in $\RR^{d+2}$ if $\Omega$ is spectral in $\RR^d$.
\end{enumerate-math}
Using this, we can now prove Theorem  \ref{main:spectralnottile}:

\begin{proof}[Proof of Theorem \ref{main:spectralnottile}.]
By \cite[Theorem 1.2]{T04} and \cite[Section 3]{KM06}, if $d\geq 3$,
we can choose a finite union of  disjoint open unit cubes $\Omega\subset \RR^d$ which is spectral but does not tile by translations.
Hence, by applying the construction above we obtain an open set $\Omega_1\subset \RR^{d+2}$ which is connected,  spectral and does not tile $\RR^{d+2}$ by translations.  This proves Theorem \ref{main:spectralnottile}.
\end{proof}

\section{Connected translational tiles that are not spectral}\label{sec:tilenotspectral}
In this section we establish Theorem \ref{main:tilenotspectral}. Observe that the construction described in the previous section might not give rise to a non-spectral set, due to the lack of a converse statement to Theorem \ref{thm:general}(ii). It is an open problem whether such a converse statement is, in fact, true.

Let $\Omega$ be a bounded  open set in $\RR^d$ and let $v\in \RR^d$ be a vector. Let $u=(v,1)\in\RR^{d+1}$, $n\geq 1$. We say that the set 
\begin{equation}
    \Omega'\coloneqq \Omega\times [0,1]+\{0,u,2u,\dots, (n-1)u\} 
\end{equation}
is a \emph{stacking} of $\Omega$. See Figure \ref{fig:stacking} for a visual illustration of the notion.

\begin{figure}[ht]
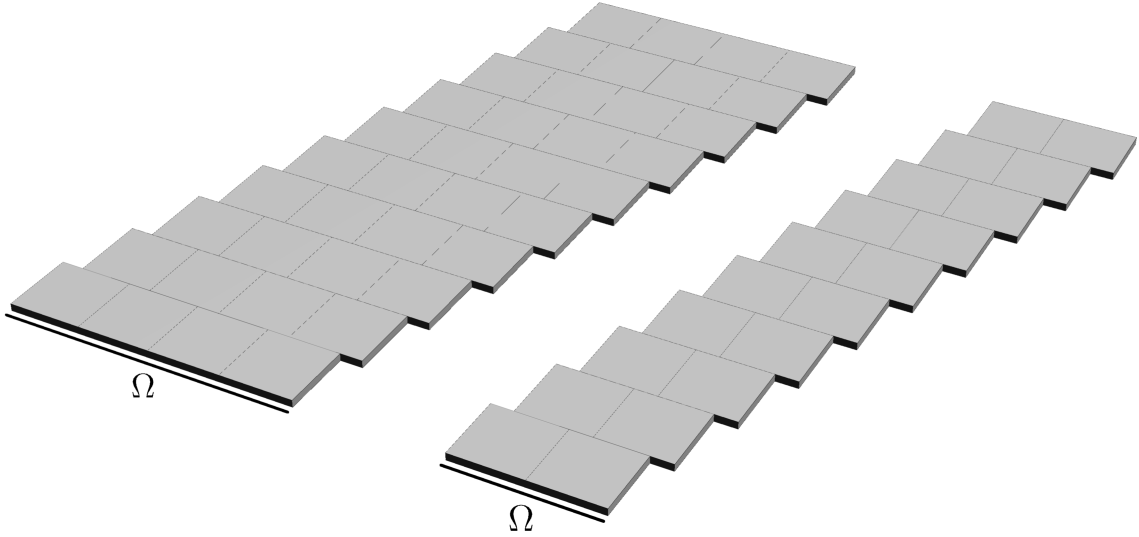

\centering
\begin{asy}
import three;
size(15cm);
currentprojection = perspective(10, -10, 10);

// Draw a square in 3d
void sq(triple p, pen color) {
    triple u=(1, 0, 0), v=(0, 1, 0);
    real h=0.1;
    path3 c=p--p+u--p+u+v--p+v--cycle;
    draw(c,black+opacity(1)+linewidth(0.02));
    draw(shift(h*cross(u, v))*c,black+opacity(1)+linewidth(0.02));
    draw(shift(p)*scale(1, 1, h)*unitcube, color);
}

int[] T={0, 1, 2, 3, 6, 7}; // 1d tile

real xoffset=0.3;
var J=10;

draw("$\Omega$", align=S, g=(T[0], -0.1, 0)--(T[3]+1, -0.1, 0), black+linewidth(1));
draw("$\Omega$", align=S, g=(T[4], -0.1, 0)--(T[5]+1, -0.1, 0), black+linewidth(1));

for(int j=0; j<J; ++j) {
 triple offset=(j*xoffset, j, 0);
 pen P=palegray+opacity(1);
 //if(j==0) P=blue+opacity(1);
 for(int i=0; i<T.length; ++i) {
  sq((T[i], 0, 0)+offset, P);
 }
}
\end{asy}
\caption{A stacking $\Omega'$ of the set $\Omega$ in one dimension higher.}
 \label{fig:stacking}
\end{figure} 

Note that by Theorem \ref{thm:general} (ii)  we have that a stacking $\Omega'$ of $\Omega$ tiles $\RR^{d+1}$ by translations if and only if $\Omega$ tiles $\RR^d$ by translations.

  In addition, we have the following:

\begin{theorem}\label{thm:stacking}
Let $\Omega$ be a measurable set in $\RR^d$ of finite measure. Suppose that $\Omega'$ is a stacking of $\Omega$.  If $\Omega' \subset \RR^{d+1}$ is spectral then   $\Omega\subset \RR^{d}$ is spectral.
\end{theorem}

\begin{remark}
    Note that Theorem \ref{thm:general}(ii) gives that the converse is also true: If $\Omega\subset \RR^d$ is spectral then $\Omega' \subset \RR^{d+1}$ is spectral. However, in this section we will only use the direction in the statement of Theorem \ref{thm:stacking}.
\end{remark}

\begin{proof}
    We have
$$
\one_{\Omega'} = \one_{\Omega\times [0,1]} * (\delta_0 + \delta_{u} + \cdots + \delta_{(n-1)u})
$$
so, when $u\cdot \xi \notin\ZZ$, with $\xi=(\xi_1, \xi_2, \ldots, \xi_{d+1})\in\RR^{d+1}$, we have
\begin{align}
\ft{\one_{\Omega'}}(\xi) &=
 \ft{\one_\Omega}(\xi_1, \ldots, \xi_d)\ft{\one_{[0,1]}}(\xi_{d+1}) \left( \sum_{j=0}^{n-1} e^{2\pi i j (u\cdot\xi)}\right) \nonumber\\
&= \ft{\one_\Omega}(\xi_1, \ldots, \xi_d)\ft{\one_{[0,1]}}(\xi_{d+1}) \frac{1-e^{2\pi i n (u\cdot \xi)}}{1-e^{2\pi i(u\cdot\xi)}}. \label{new-zeros}
\end{align}
(Since we care about zeros introduced beyond those of $\ft{\one_{\Omega \times[0,1]}}$ we may assume that $u\cdot\xi \notin \ZZ$ -- see below.)
Define the subgroup of $\RR^{d+1}$
$$
G = \Set{\xi=(\xi_1, \xi_2, \ldots, \xi_{d+1}): u\cdot \xi \in \frac1n \ZZ }
$$
and its subgroup of index $n$
$$
H = \Set{\xi=(\xi_1, \xi_2, \ldots, \xi_{d+1}): u\cdot \xi \in \ZZ}.
$$
From \eqref{new-zeros} it follows that the zeros of $\ft{\one_{\Omega'}}$ are those due to $\ft{\one_{\Omega\times [0,1]}}$ plus the union of cosets of $H$ in $G$
$$
D = \left(H+\frac{u}{n\|u\|^2}\right) \cup \left(H+\frac{2u}{n\|u\|^2}\right) \cup \ldots \cup \left(H+\frac{(n-1)u}{n\|u\|^2}\right).
$$
If two distinct points of $\RR^{d+1}$ are in the same coset of $H$ then their difference is in $H$, so it is not in $D$.

Suppose $\Lambda' \subseteq \RR^{d+1}$ is a spectrum of $\Omega'$.
Then, by Proposition \ref{spectrality}: $$\dens{\Lambda'} = \Abs{\Omega'}=n\Abs{\Omega}=n\Abs{\Omega\times [0,1]}$$
(since for every $0\leq j < j'\leq n-1$, $|(\Omega\times [0,1]+ju)\cap (\Omega\times [0,1]+j'u)|=0$).
We will now select elements of $\Lambda'$ of density at least $\Abs{\Omega\times [0,1]}$  whose pairwise differences do not intersect $D$. If we call $\Lambda$ the set of  those elements of $\Lambda'$ that we kept, it follows that the pairwise differences of $\Lambda$ all fall in $\Set{\ft{\one_{\Omega\times [0,1]}}=0}$.

To select the points of $\Lambda'$, we want we look at every coset $\lambda+G$, $\lambda\in\Lambda'$. For each $\lambda_0\in\Lambda'$, at least a fraction $1/n$ of the points in $\Lambda'\cap \lambda_0+G$ are on 
one of the cosets
\begin{equation}\label{Hcosets}
    \lambda_0+H +j\tilde u,\quad  j=0, 1, \dots, n-1
\end{equation}
of $H$, where $$\tilde u=\frac{u}{n\|u\|^2}.$$ We keep precisely those points of $\Lambda'$ on $\lambda_0+G$, i.e., those on the most populated (highest density) of the  $n$ cosets \eqref{Hcosets}.
It follows that for any two points we kept their difference is either not in $G$ (hence also not in $D \subseteq G$) or, if their difference is in $G$, then it is in $H$, hence again not in $D$. 

Thus,  we conclude that if $\Lambda'$ is a spectrum for $\Omega'$ then we have that $E(\Lambda)$ is orthogonal in $L^2(\Omega\times [0,1])$. Moreover, by the construction of $\Lambda$ its lower   density  is bounded from below by $$\frac1n \dens{\Lambda'}= |\Omega\times [0,1]|.$$
By Proposition \ref{spectrality}, we then have that   $\Omega\times [0,1]$ is spectral. Thus, from  \cite[Theorem 1.1]{GL16} it follows  that $\Omega$ is spectral.
\end{proof}

Using Theorem \ref{thm:stacking}, we can finally prove Theorem \ref{main:tilenotspectral}:

\begin{proof}[Proof of Theorem \ref{main:tilenotspectral}.]
Let $d\geq 3$. By \cite{KM2}, we can choose a finite disjoint union of unit cubes $\Omega\subset \RR^d$ which tiles the space by translations and is not spectral.

Our goal is to construct higher dimensional bridges  between the connected components of $\Omega$ while preserving its tiling and spectral properties.

We denote by $C_1,\dots , C_m\subset \Omega$ , $m>1$, the connected components of $\Omega$, and let $\tilde C_j$ be the set of the centers of the cubes that $C_j$ consists of. We may assume, without the loss of generality, that 
$$
\min_{1\leq i<j\leq m} \min\{\|c_i-c_j\|\colon c_i\in \tilde C_i, c_j\in \tilde C_j\}=\min\{\|c_1-c_2\|\colon c_1\in \tilde C_1, c_2\in \tilde C_2\}.
$$

Let 
$$D_{(C_1,C_2)}\coloneqq \min\{\|c_1-c_2\|\colon c_1\in \tilde C_1, c_2\in \tilde C_2\}=\|b-a\|$$
where  $a\in \tilde C_1$, $b\in \tilde C_2$  are centers of unit cubes in $C_1$, $C_2$ of minimal distance.

Let $n=\Ceil{D_{(C_1,C_2)}}$ be the natural number closest (from above) to $\|b-a\|$, so that 
\begin{equation}\label{eq:connect}
    |C_j\cap (C_j+v)|>0,\quad j=1,\dots,m
\end{equation}
where $v=\frac{(b-a)}{n}\in \RR^d$. Consider the stacking $\Omega_1$ of $\Omega$:
\begin{equation}
    \Omega_1:=\Inn\left((\Omega\times [0,1])\oplus\Set{0,u,2u,\dots, \Floor{\frac{n}{2}} u}\right)
\end{equation}
where $u=(v,1)\in \RR^{d+1}$. This is a disjoint sum because of the 1 in the last coordinate of $u$, up to measure zero. In other words, for every $0\leq j<j'\leq \Floor{\frac{n}{2}}$ $$|(\Omega\times [0,1]+ju)\cap (\Omega\times [0,1]+j'u)|=0.$$ By Theorem \ref{thm:general} (since the boundary of $\Omega$ has zero measure) we know that $\Omega_1$ tiles $\RR^{d+1}$ since $\Omega$ tiles $\RR^d$ and from Theorem \ref{thm:stacking} we also have that  $\Omega_1$ is not spectral since  $\Omega$ is not spectral. 
We denote 
$$C^1_j=\Inn\left( (C_j\times [0,1])\oplus \Set{0,u,2u,\dots, \Floor{\frac{n}{2} u}}\right),\quad j=1,\dots,m.$$
Then, by \eqref{eq:connect}, for each $j=1,\dots,m$, the set $C^1_j$ is connected; moreover, its closure is a finite union of closed unit cubes with centers
$$
\tilde C^1_j:=\pp{\tilde C_j\times\Set{\frac{1}{2}}}\oplus \Set{0,u,2u,\dots, \Floor{\frac{n}{2}} u}.
$$

Let 
$$
D_{(C^1_1,C^1_2)}\coloneqq \min\Set{\Norm{c^1_1-c^1_2}\colon c^1_1\in \tilde C^1_1, c^1_2\in \tilde C^1_2}.
$$
Observe that, as 
$$\pp{a,\frac{1}{2}}+\Floor{\frac{n}{2}} u\in C^1_1\quad 
 \pp{b,\frac{1}{2}}\in C^1_2,$$ and 
$$n-1<\Norm{a-b}=D_{C_1,C_2}\leq n,
$$
we have:
\begin{align}
     D_{(C^1_1,C^1_2)}&\leq \Norm{\pp{a,\frac{1}{2}}+\Floor{\frac{n}{2}} u-\pp{b,\frac{1}{2}}} \notag \\
    &=  \Norm{\pp{\pp{a-b}+\Floor{\frac{n}{2}}\frac{\pp{b-a}}{n},\Floor{\frac{n}{2}}}} \notag \\
     &= \frac{1}{2} \begin{cases} 
    \Norm{\pp{\frac{n+1}{n}(a-b),n-1}} & n \hbox{ is odd}\\
    \Norm{\pp{a-b,n}} & n \hbox{ is even}
    \end{cases} \notag \\ 
    &= \frac{1}{2}\begin{cases} 
    \sqrt{\pp{\frac{n+1}{n}\Norm{a-b}}^2+\pp{n-1}^2} & n \hbox{ is odd}\\
    \sqrt{\Norm{a-b}^2+n^2 } & n \hbox{ is even}
    \end{cases}
    \notag \\
    &\leq \frac{1}{2}\begin{cases} 
    \sqrt{2\pp{\frac{n+1}{n}\Norm{a-b}}^2} & n \hbox{ is odd}\\
    \sqrt{2n^2 } & n \hbox{ is even}
    \end{cases}
    \notag \\
    &=  \frac{1}{\sqrt 2} \begin{cases} 
    \frac{n+1}{n} D_{(C_1,C_2)} & n \hbox{ is odd}\\
    n & n \hbox{ is even}
    \end{cases}. \label{eq:distance}
\end{align}

\begin{figure}[ht]
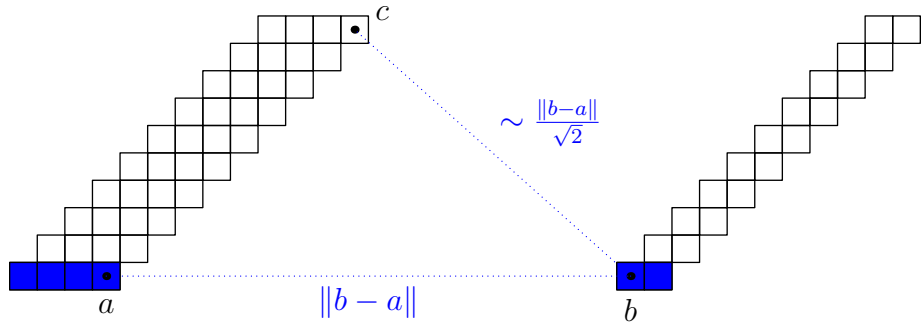

\centering
\begin{asy}
size(12cm);

int i, j, k, NN=10;

int[] T={0, 1, 2, 3, 22, 23}; // 1d tile

pair a=(T[3], 0)+(0.5, 0.5), b=(T[4], 0)+(0.5, 0.5), c=a+(NN-1, NN-1);

for(i=0; i<NN; ++i) {
	for(j=0; j<T.length; ++j)
		draw(box((T[j]+i, i), (T[j]+1+i, i+1)));
}
for(j=0; j<T.length; ++j) filldraw(box((T[j], 0), (T[j]+1, 1)), blue);
dot(a); label("$a$", a-(0,1/2), S);
dot(b); label("$b$", b-(0,1/2), S);
draw("$\|b-a\|$", a--b, blue+dotted);
dot(c); label("$c$", c+(1,0), N);
draw("$\sim \frac{\|b-a\|}{\sqrt{2}}$", b--c, blue+dotted, align=NE);
\end{asy}
\caption{Shortening the distance between two connected components. The slope of the line from $a$ to $c$ is approximately 1 when $n$ is large. The blue set is $\Omega\times[0, 1]\subseteq \RR^d\times\RR$.}
 \label{fig:shorten}
\end{figure}

We have the following possible cases:
\begin{itemize}
\item[Case 1:] If $D_{(C_1,C_2)}< 2$, then $n\leq 2$ and we have that $\Inn( C^1_1\cup C^1_2)$ is connected. Indeed, clearly  $D_{(C_1,C_2)}> 1$ as otherwise   $C_1\cap C_2$ is non-empty but this contradicts the assumption that $C_1, C_2$ are different connected components; therefore, we must have $n= 2$,  $u=(\frac{b-a}2, 1)$ and the intersection of the cube in $C^1_1$ that is centered at $(\frac{a+b}{2},\frac{3}{2})$  and the side of the cube centered at $(b,\frac{1}{2})$ in $C^1_2$ contains interior points. Thus, $\Omega_1$ has at most $m-1$ connected components. (See Figure \ref{fig:touching}.)

\begin{figure}[ht]
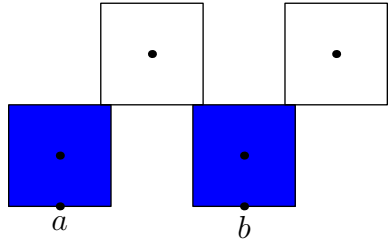

\centering
\begin{asy}
size(5cm);

real eps=0.2, a=0.5, b=a+2-eps;

filldraw(box((0, 0), (1, 1)), blue);
filldraw(box((b-0.5, 0), (b+0.5, 1)), blue);
dot((a, 0.5)); dot((a, 0)); label("$a$", (a, 0), S);
dot((b, 0.5)); dot((b, 0)); label("$b$", (b, 0), S);
draw(box(((a+b)/2-0.5, 1), ((a+b)/2+0.5, 2)));
dot(((a+b)/2, 3/2));
draw(box(((a+b)/2-0.5+(b-a), 1), ((a+b)/2+0.5+(b-a), 2)));
dot(((a+b)/2+(b-a), 3/2));
\end{asy}
\caption{This is Case 1, with $D_{(C_1, C_2)} < 2$.}
 \label{fig:touching}
\end{figure} 
    \item[Case 2:] If $D_{(C_1,C_2)}\geq 2$, then, by \eqref{eq:distance}: 
    \begin{equation}\label{eq:contraction}
        D_{(C^1_1,C^1_2)}< \frac{4}{3\sqrt 2}D_{(C_1,C_2)} <  0.94281 \cdot D_{(C_1,C_2)}.
    \end{equation} 
    Indeed, if $D_{(C_1,C_2)}= n=2$, then \eqref{eq:distance} gives
    $$D_{(C^1_1,C^1_2)}\leq \frac{D_{(C_1,C_2)}}{\sqrt 2}
    $$ which implies \eqref{eq:contraction}.
    If  $3\leq n$ is odd, then  \eqref{eq:distance}  implies \eqref{eq:contraction}, since 
    $$\frac{(n+1)}{n}\leq \frac{4}{3}
    $$ in this case.
Otherwise,  $4\leq  n$ is even, and then by \eqref{eq:distance} we have
$$
D_{(C_1^1, C_2^1)} \le \frac{n}{\sqrt2} \le \frac{D_{(C_1,C_2)}+1}{\sqrt2},
$$
which implies \eqref{eq:contraction} since
$$D_{(C_1,C_2)}+1< \frac{4}{3}D_{(C_1,C_2)}
$$ in this case.
\end{itemize}

Unless Case 1 applies, we repeat the process above. In the $k$-th iteration ($k\geq 2$),
the distance between the components $C^{k-1}_1$ and $C^{k-1}_2$ of $\Omega_{k-1}\subset \RR^{d+k-1}$ shrinks at a uniform rate in $C^k_1,C^k_2\subset \Omega_k\subset \RR^{d+k}$. (See Figure \ref{fig:shorten}.)
Hence, after $l<\infty$ iterations, we obtain a set $\Omega_l$ in $\RR^{d+l}$ which is a tile and is not spectral and such that
$D_{(C^{l-1}_1,C^{l-1}_2)}< 2$. Therefore, as in Case 1 above, the set $C^l_1\cup C^l_2$  in $ \Omega_l\subset \RR^{d+l}$ is connected. We constructed  a ``spiral bridge'' in $\Omega_l\subset \RR^{d+l}$ between the original components $C_1$ and $C_2$ of $\Omega$, thus $\Omega_l$  has at most $m-1$ connected components. 

\begin{figure}[ht]
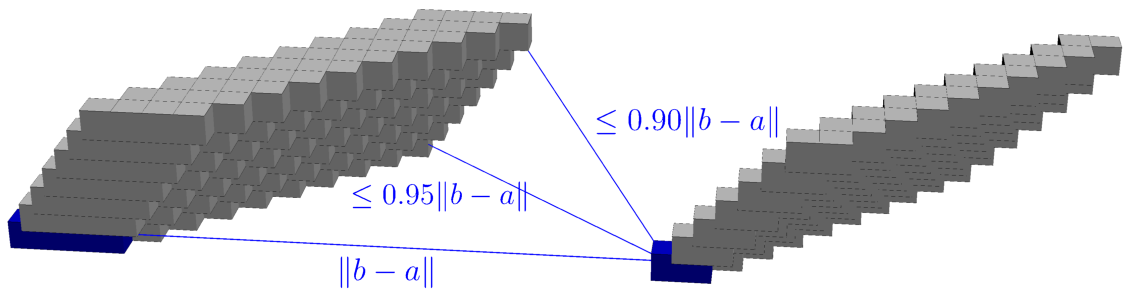

\centering
\begin{asy}
import three;
size(15cm,0);

currentprojection = perspective(10, 5, -1);

void thick(triple p, triple u, triple v, real thickness, pen color) {
    path3 c=p--p+u--p+u+v--p+v--cycle;
    draw(c, black+opacity(1)+linewidth(0.02));
    draw(surface(c),color);
    triple n=scale3(thickness)*unit(cross(u, v));
    c=shift(n)*c;
    draw(c,black+opacity(1)+linewidth(0.02));
    draw(surface(c),color);

    path3 t=p--p+n--p+n+u--p+u--cycle;
    draw(surface(t), color);
    draw(surface(shift(v)*t), color);

    path3 t=p--p+n--p+n+v--p+v--cycle;
    draw(surface(t), color);
    draw(surface(shift(u)*t), color);
}

int i, j, k, z, NN=10, NNN=7;

int[] T={0, 1, 2, 3, 22, 23}; // 1d tile

triple a=(T[3], 0, 0.5)+(0.5, 0.5, 0), b=(T[4], 0, 0.5)+(0.5, 0.5, 0), c=a+(NN-1, NN-1, 0),
  v=(b-c)/(2*NNN), d=c+(NNN-1)*v+(0, 0, NNN-1);

pen P=palegray+opacity(1), Q=blue+opacity(1);

real L = T[T.length-1];
currentprojection = perspective(0.1*L, -L/4, 3);

for(z=0; z<NNN; ++z) {
 for(i=0; i<NN; ++i) {
	for(j=0; j<T.length; ++j)
		thick( (T[j]+i, i, z)+z*v, (1, 0, 0), (0, 1, 0), 1, (i==0 && z==0)?Q:P);
 }
}
draw("$\|b-a\|$", a--b, blue, align=S);
//draw("$\sim \frac{\|b-a\|}{\sqrt{2}}$", c--b, blue, align=W);
//draw("$\sim \frac{\|b-a\|}{2}$", d--b, blue, align=NE);
draw("$\le 0.95\|b-a\|$", c--b, blue, align=W);
draw("$\le 0.90\|b-a\|$", d--b, blue, align=NE);

\end{asy}
\caption{Two steps of the stacking procedure. The blue set is $\Omega\times[0, 1]^2 \subseteq \RR^d\times\RR^2$. The distance between the two connected components is being reduced exponentially.
The bottom layer of the cubes is the same as that in Fig. \ref{fig:shorten}, where they are shown in dimension $d+1$ (after just one step of the stacking procedure.}
 \label{fig:3dbridge}
\end{figure} 

We iterate this process, constructing  $m-1$ spiral bridges between all the components of the original set $\Omega$ while preserving its tiling  and non-spectrality properties, to eventually obtain a connected set $\tilde \Omega\subset \RR^{\tilde d}$ which tiles the space by translations and is not spectral. Finally, observe that by construction, $\tilde \Omega$ is a finite union of closed unit cubes, hence $\tilde \Omega$ is the closure of its interior. This completed the proof of Theorem \ref{main:tilenotspectral}.
\end{proof}

\section{Discussion and open problems}\label{sec:open}

\subsection{Repairing the periodic tiling conjecture}
Despite the fact that several positive results towards Conjecture \ref{ptc} have been obtained over the years (see \cite[Section 1]{GT22} for a partial list), the conjecture was recently proven to be false in high dimensions \cite{GT22}. However, the aperiodic translational tile constructed  in \cite{GT22}  is a very complicated disconnected set, and, on the other hand, Conjecture \ref{ptc} is known to hold for convex domains in {\it all} dimensions \cite{V,M} in a strong sense: every convex translational tile is also a {\em lattice} tile. This naturally motivates one to seek the weakest regularity assumption on the structure of a set under which  the periodic tiling conjecture is true in all dimensions. 

In this paper we construct  aperiodic translational tiles which are connected, showing that a connectedness assumption is not strong enough for the purpose of repairing the periodic tiling conjecture. We therefore must  strengthen it, and look for a regularity assumption in the spectrum between connectedness and convexity. This gives rise to the following questions:

\begin{question}\label{simply}
    Does Conjecture \ref{ptc} hold for simply connected sets in all dimensions?  
\end{question}

We suspect that by adapting the method in this paper, constructing folded bridges between the  connected components, one might prove a negative answer to Question \ref{simply}.  Upon a negative answer to Question \ref{simply}, we can further ask:

\begin{question}\label{balls}
    Does Conjecture \ref{ptc} hold for topological balls in all dimensions?  
\end{question}

Note that while Conjecture \ref{ptc} is still open in the plane\footnote{But is known to be true in $\ZZ^2$ \cite{BH}.}, it is known to be true for topological disks \cite{BN,K}.

\subsection{Repairing Fuglede's conjecture}
Conjecture \ref{Fug} inspired extensive research concerning the connection between spectrality and   tiling by translations.   Over time, it has became apparent that in many respects, spectral sets ``behave like" sets which can tile the space by translations. 
However, after a few decades, counterexamples to both directions of the conjecture were constructed in dimension $d\geq 3$ (see \cite[Section 4]{KM10} and the references therein).

Although the connection between the analytic notion of spectrality and the geometric notion of tiling by translations has been intensively studied, the precise connection is still a mystery.

\begin{question}\label{ques:spectile}
What is the precise connection between spectral sets and translational tiles?
\end{question}

This suggests the problem of determining  the exact conditions under which  Conjecture \ref{Fug} holds. 
In this paper, we solve the problem for {\it connected} sets, showing that there are connected counterexamples to Fuglede's conjecture.
On the other hand, Conjecture \ref{Fug} was proven to hold for convex domain in all dimensions \cite{IKT,GL17,LM}. This suggests the study of the following question:

\begin{question}
Are there any topological conditions on a set that force either of the directions of the Conjecture \ref{Fug} to be true?
\end{question}

\subsection{Connectedness in low dimensions}
Our main results, Theorems \ref{main:aperiodic}, \ref{main:spectralnottile} and \ref{main:tilenotspectral}, demonstrate that the higher the dimension is the weaker a connectedness assumption becomes. In particular, we show that any aperiodic $d$-dimensional translational tile gives rise to a $(d+2)$-dimensional aperiodic {\it connected} translational tile. One can ask about the necessity of the two additional dimensions, as follows: 

\begin{question}
What is the minimal $d$ such that there is a $d$-dimensional {\em connected} aperiodic translational tile? 
\end{question}

We can ask the corresponding questions in the context of Conjecture \ref{Fug}:

\begin{question}
What is the minimal $d\leq  5$ such that there is a $d$-dimensional {\em connected} counterexample to the direction ``spectral $\Rightarrow$ tiles'' of Conjecture \ref{Fug}?   
 \end{question}

\begin{question}
What is the minimal $d$ such that  there is a $d$-dimensional {\em connected} counterexample to the direction ``tiles $\Rightarrow$ spectral'' of Conjecture \ref{Fug}? 
\end{question}

In particular, can the proof of Theorem \ref{main:tilenotspectral} be amended to give a connected, non-spectral tile in a known dimension, in the spirit of Theorem \ref{main:spectralnottile}?
 Notice that the construction in the proof of Theorem \ref{main:tilenotspectral} of spiral bridges goes up in dimension by a number that depends on the tile we are starting from.

\subsection{Aperiodicity and spectrality}
In \cite{fug} it was observed that by the Poisson summation formula, for a {\em lattice} $\Lambda\subset \RR^d$, a measurable set $\Omega\subset \RR^d$ tiles by translations along $\Lambda$ if and only if the dual lattice $\Lambda^*$ is a spectrum for $\Omega$. This might be regarded as the motivation for  Conjecture \ref{Fug}. Thus, the recent discovery of aperiodic translational tiles \cite{GT22} brings up the question about possible connection between counterexamples to Conjecture \ref{ptc} and counterexamples to Conjecture \ref{Fug}:

\begin{question}\label{aperiodicspec}
Is there any aperiodic translational tile $\Omega\subset \RR^d$ which is spectral?
\end{question}

Note that a negative answer to Question \ref{aperiodicspec} would give rise to a new class of counterexamples to Fuglede's conjecture.

\subsection{Quantitative aperiodicity in dimension 1}

It is well known that if  a finite $F \subseteq \ZZ$ tiles $\ZZ$ by translation then the tiling is necessarily periodic \cite{newman}. In other words if $F \oplus A = \ZZ$ then there is $N>0$ such that $A+N = A$. How large can or must this $N$ be compared to a measure of size of $F$, let us say compared to the diameter $D$ of $F$? While it is known that $N$ can be even exponentially large in $D$ \cite{kolountzakis2003long,steinberger05,steinberger09} , and  must be at most polynomially large in $D$, when $|F|$ is kept fixed \cite{GT21a}, no example of a tile $F$ is known where the minimal possible such $N$ (over all possible tilings by $F$) is more than linearly large in $D$. Such a tile $F$, all of whose tilings by translation would have periods much larger than $D$, would be a one-dimensional, quantitative analogue of aperiodicity.

\begin{question}
Does there exist a family of finite sets $F_n\subseteq\ZZ$ with diameter $$\diam{F_n} \to \infty$$ which tile by translation and  the minimal period $N_n$ of the tilings that $F_n$ admits satisfies
$$
\frac{N_n}{\diam{F_n}} \to \infty?
$$
\end{question}

Observe that the existence of such a construction   would, in particular, disprove the Coven--Meyerowitz conjecture \cite{CM}. (See also the discussion in \cite[Section 4]{LZ}.) 

\subsubsection*{Note added in proof}  Izabella \L aba and Dmitrii Zakharov \cite{LZ} recently  proved that there exists and absolute constant $c>0$ such that if a finite set $F\subset \ZZ$ tiles $\ZZ$ then it must admit a tiling of period at most $$\exp\left(\frac{c(\log \diam{F})^2}{\log \log \diam{F}}\right).$$

\end{document}